\documentclass[12pt]{amsart}
\usepackage{amsmath}
\usepackage{amssymb}
\usepackage{amsfonts}
\usepackage{amsthm}
\usepackage{mathrsfs}
\usepackage[all]{xy}
\usepackage{color}
\usepackage{multicol}
\usepackage{enumerate}
\usepackage{bbm}

\usepackage[vcentermath,enableskew]{youngtab}
\usepackage{multirow}
\usepackage{tikz}
\usetikzlibrary{arrows}
\usepackage{ytableau}

\usepackage[colorlinks=true, pdfstartview=FitV, linkcolor=blue, citecolor=blue, urlcolor=blue]{hyperref}
\theoremstyle{plain}
\newtheorem*{Stkthm}{The Stacking Theorem}

\newtheorem{theorem}{Theorem}[section]
\newtheorem{lemma}[theorem]{Lemma}
\newtheorem{cor}[theorem]{Corollary}

\newtheorem{proposition}[theorem]{Proposition}
\newtheorem{rem}[theorem]{Remark}
\newtheorem{thm}[theorem]{Theorem}
\newtheorem{lem}[theorem]{Lemma}

\theoremstyle{definition}

\newtheorem{example}[theorem]{Example}

\newcommand{\Rad}{\mathrm{Rad}}

\newcommand{\C}{\mathbb{C}}

\newcommand{\I}{\mathbb{I}}
\renewcommand{\k}{\mathbbm{k}}

\newcommand{\B}{\mathcal{B}}

\newcommand{\cN}{\mathcal{N}}
\newcommand{\cO}{\mathcal{O}}

\newcommand{\cX}{\mathcal{X}}

\renewcommand{\O}{\mathcal{O}}

\newcommand{\y}{10}
\newcommand{\ya}{11}
\newcommand{\yb}{12}
\newcommand{\yc}{13}
\newcommand{\yd}{14}
\newcommand{\ye}{15}

\DeclareMathOperator{\GL}{GL}
\DeclareMathOperator{\Std}{Std}
\DeclareMathOperator{\stk}{\bf stk}

\DeclareMathOperator{\nil}{nil}
\DeclareMathOperator{\LS}{\operatorname{\bf LS}}

\DeclareMathOperator{\Span}{span}

\DeclareMathOperator{\Type}{Type}
\DeclareMathOperator{\Comp}{Comp}

\DeclareMathOperator{\Lie}{Lie}

\newcommand{\ub}{\underline{\b}}

\newcommand{\oF}{\bar{F}}

\newcommand{\g}{\mathfrak{g}}
\newcommand{\gl}{\mathfrak{gl}}

\renewcommand{\t}{\mathfrak{t}}
\renewcommand{\b}{\mathfrak{b}}
\newcommand{\n}{\mathfrak{n}}

\newcommand{\Ad}{\operatorname{Ad}}

\newcommand{\p}{\mathfrak{p}}

\newcommand{\Ind}{\operatorname{Ind}}

\newcommand{\codim}{\mathrm{codim}}

\def\into{\hookrightarrow}
\def\onto{\twoheadrightarrow}
\def\isoto{\overset{\sim}{\longrightarrow}}

\numberwithin{equation}{section}

\begin{document}
\title[Parabolic Induction for Springer fibres]{Parabolic Induction for Springer fibres}

\author{Lewis Topley}
\address{Department of Mathematical Sciences, University of Bath, Claverton Down, Bath, BA2 7AY, UK}\email{lt803@bath.ac.uk} 

\author{Neil Saunders}
\address{N.~Saunders: School of Computing and Mathematical Sciences, University of Greenwich, Park Row, London, SE10 9LS}
\email{n.saunders@greenwich.ac.uk}

\maketitle

\begin{abstract}
Let $G$ be a reductive group satisfying the standard hypotheses, with Lie algebra $\g$. For each nilpotent orbit $\O_0$ in a Levi subalgebra $\g_0$ we can consider the induced orbit $\O$ defined by Lusztig and Spaltenstein. We observe that there is a natural closed morphism of relative dimension zero from the Springer fibre over a point of $\O_0$ to the Springer fibre over $\O$, which induces an injection on the level of irreducible components. When $G = \GL_N$ the components of Springer fibres was classified by Spaltenstein using standard tableaux. Our main results explains how the Lusztig--Spaltenstein map of Springer fibres can be described combinatorially, using a new associative composition rule for standard tableaux which we call {\it stacking}.
\end{abstract}

\section{Introduction}

Modern representation theory is a harmonious confluence of algebra, geometry and combinatorics. There are now many examples of algebraic problems which exhibit complementary geometric and combinatorial solutions. In this article we introduce a natural morphism between certain algebraic varieties which are ubiquitous in representation theory, and we describe this morphism combinatorially, thus continuing the trialogue between these fields.

The theory of Lusztig--Spaltenstein (LS) induction is a geometric process for constructing nilpotent orbits in the Lie algebra of a reductive group from the nilpotent orbits in Levi subalgebras. Nilpotent orbits are one of the most pervasive geometric objects appearing in representation theory, taking a starring role in:
\begin{enumerate}
\item Springer's construction of Weyl group representations \cite{CG}
\item associated varieties of primitive ideals of enveloping algebras \cite[\textsection 9]{JaNO}.
\item conical symplectic singularities \cite{Bea}.
\item modular representation theory of Lie algebras \cite{BMR}.
\end{enumerate}
It is therefore unsurprising that LS induction has played a important role in these theories, serving as a geometric analogue of parabolic induction of representations.

Let $G$ be a reductive algebraic group over an algebraically closed field $\k$ of any characteristic and assume the standard hypotheses \cite[\textsection 2.9]{JaNO}. Perhaps the most famous projective variety appearing in representation theory is the flag variety $G/B$ of a reductive algebraic group $G$. This is the base space of the Springer resolution $T^* G/B$, and can be identified with the set of Borel subalgebras in $\g$. The fibres of the resolution $T^* G/B \to \cN(\g)$ are the the {\it Springer fibres}: for $e\in \cN(\g)$ the fibre is described as
$$\B_e := \{gB \in G/B \mid e\in \Ad(g) \b\}$$
They carry a wealth of representation theoretic information, placing them at the very core of geometric representation theory. We begin this paper by showing that LS induction of orbits induces a morphism of relative dimension zero between the corresponding Springer fibres (Proposition~\ref{P:LSwelldefined}). Let us state this more precisely. If $\g_0 \subseteq \p \subseteq \g$ is a Levi factor of a parabolic and $\O_0 \subseteq \cN(\g_0)$ is an orbit then for $e_0 \in \O_0$ and $e \in (e_0 + \n) \cap \Ind_{\g_0}^\g(\O_0)$ we define a morphism
\begin{eqnarray}
\label{e:LSintro}
\LS : \B_{e_0}^0 \longrightarrow \B_e
\end{eqnarray}
where $\B^0 $ is the flag variety for $G_0$. Since this definition is elementary this construction is probably well-known to experts, however we cannot find a reference in the literature.

When $G = \GL_N$ the theory of Springer fibres is especially nice. Spaltenstein showed that there is a map $\Phi : \B_e \to \Std(\lambda)$ where $\lambda \vdash N$ and $\Std(\lambda)$ denotes the set of standard tableaux of shape $\lambda$. For $\sigma \in \Std(\lambda)$ the closure $C_\sigma:= \overline{\Phi^{-1}(\sigma)}$ is an irreducible component of $\B_e$. Furthermore the combinatorics of the tableaux control aspects of the geometry of the sets $\Phi^{-1}(\sigma)$ and their closures \cite{Spa, FM}. We abuse notation viewing $\Phi$ as a bijection $\Comp \B_e \to \Std(\lambda)$.

If $\g_0$ is a Levi subalgebra then $\g \cong \gl_{\lambda_1} \times \cdots \times \gl_{\lambda_n}$ for some $\lambda \vDash N$, and the nilpotent orbits in $\g_0$ are classified by a tuple of partitions $(\mu_1,...,\mu_n)$ with $\mu_i \vdash \lambda_i$. If $\O_0 \subseteq \cN(\g_0)$ is classified by $(\mu_1,...,\mu_n)$ then Spaltenstein's construction gives a bijection $\Phi_0 : \Comp \B_{e_0}^0 \to \prod_{i=1}^n \Std(\mu_i)$.

The purpose of this paper is to describe the combinatorics lying behind the Lusztig--Spaltenstein morphism \eqref{e:LSintro}. In Section~\ref{ss:combinatorics} we introduce a new map $\stk : \prod_{i=1}^n \Std(\mu_i) \to \Std(\lambda)$, combinatorially defined, which we call {\it the stacking map}. This is our main result (Cf. Theorem~\ref{T:stackwithproof}).

\begin{Stkthm}
The combinatorics of components of Springer fibres under $\LS$ is determined by the stacking map, so the following diagram commutes
\begin{center}
\begin{tikzpicture}[node distance=2.8cm, auto]
 \node (A) {$\Comp \B^0_{e_0}$};
 \node (B) [below of= A] {$\prod_{i=1}^n \Std(\mu_i)$};
 \node (C) [right of= A] {$\Comp \B_e$};
  \node (D) [right of= B]{$\Std(\lambda).$};
 \draw[->] (A) to node [swap]{$\Phi_0$} (B);
 \draw[->] (A) to node {$\LS$} (C);
 \draw[->] (B) to node [swap]{$\stk$} (D);
 \draw[->] (C) to node {$\Phi$} (D);
\end{tikzpicture}
\end{center}
\end{Stkthm}

Springer fibres bear a very close relationship to orbital varieties (see \cite[\textsection~9]{JaNO}, for example). In \cite{FMe} Fresse and Melnikov explain that Lusztig--Spaltenstein induction leads to a morphism of orbital varieties, similar to our construction. It is natural to expect that our induction of Springer fibres is compatible with their construction.

\subsection*{Acknowledgements} Many of the ideas which we present in this paper were initiated at the workshop ``Springer fibres and Geometric Representation Theory'', organised by the authors, and held at the University of Greenwich in 2019. The authors would like to thank the Heilbronn Institute for Mathematical Research for their generous support which made that event possible, and the participants for sharing their insights. The first author's research is supported by UKRI grant MR/S032657/1. The second author benefited from the LMS Scheme 4, grant number 42037.

\section{Combinatorics }

\subsection{Partitions, compositions and Young tableaux}
\label{ss:combinatorics}

Throughout the paper $N > 0$. Let $\lambda = (\lambda_1,...,\lambda_n)$ be a tuple of positive integers. If $\sum_i \lambda_i = N$ then we write $\lambda \vDash N$ and call $\lambda$ a {\it composition of $N$}, and if $\lambda$ is a non-increasing ordered composition then we write $\lambda \vdash N$ and say that $\lambda$ is a {\it partition of $N$}. We write $\ell(\lambda) = n$ for the length of $\lambda$. Denote the set of partitions of $N$ by $\mathcal{P}_{N}$ \\

When $\lambda \vdash N$, a {\it Young diagram of shape $\lambda$} is an array of boxes where the $i$-th row consists of $\lambda_{i}$ boxes (the top row is the first and the bottom row is the $n$-th). A {\it standard tableau} of shape $\lambda$ is a filling of a Young diagram of shape $\lambda$ with the alphabet $\{1, 2, \ldots, N\}$ in such a way that the numbers are increasing along rows and down columns. The set of standard tableaux of shape $\lambda$ is denoted $\Std(\lambda)$. For $T^{\lambda} \in \Std(\lambda)$ we write $t_{i,j}^{\lambda}$ for the $(i,j)$ entry of $T^{\lambda}$, for $1 \leq i \leq \ell(\lambda)$ and $1 \leq j \leq \lambda_{i}$. 

\subsection{Ordering on tableaux} \label{subsection:ordering} For this section, we follow the notational convention of Spaltenstein \cite{Spa}. Fix $\lambda \vdash n$ and let $\sigma$ denote a standard Young tableau of shape $\lambda$. For each $i \in \{1,\ldots, n\}$, let $\sigma_i$ denote the column of $\sigma$ in which the entry $i$ occurs. Observe that $\sigma$ is completely determined by the sequence $\sigma_1, \ldots, \sigma_n$. \\

For $\sigma, \tau \in \Std(\lambda)$ declare that $\sigma < \tau$ if for some $1 \leq i \leq n$, we have $\sigma_i < \tau_i$ and for each $i \leq j \leq n$, $\sigma_j=\tau_j$. This defines a total ordering on $\Std(\lambda)$. 

\begin{example}
For $\lambda=(3,2)$, the following represents the total ordering:
$$
\young(123,45)< \young(124,35)< \young(134,25)<  \young(125,34) < \young(135,24)
$$
\end{example}

\subsection{Stacking tableaux}
Let $\lambda = (\lambda_1, \lambda_2, \cdots, \lambda_n) \vDash N$. Now for each $i =1,...,n$ we pick a partition $\mu_i = (\mu_{i,1},...,\mu_{i, m_i})$ of $\lambda_i$ and write $m := \max m_i$. Define a new partition
\begin{eqnarray}
\label{e:musigmadefn}
\mu^\Sigma = (\mu^\Sigma_1, \mu^\Sigma_2, ..., \mu^\Sigma_{m}),\\
\mu^\Sigma_j := \sum_{i=1}^m \mu_{i, j}.
\end{eqnarray}
where we adopt the convention $\mu_{i,j} = 0$ for $j > m_i$.\\

Now we define a map
\begin{eqnarray}
\stk : \prod_{i=1}^n \Std(\mu_i) \into \Std(\mu^\Sigma)
\end{eqnarray}
which we call {\it the stacking map}. This is the key new combinatorial construction of this paper. \\

Now for $\lambda \vDash N$ and $\mu_{i} \vdash \lambda_i$ as above, let $(T^{\mu_{1}}, T^{\mu_{2}}, \ldots, T^{\mu_{n}})$ be a tuple of tableaux in $\prod_{i=1}^n \Std(\mu_i)$. We describe the image of $(T^{\mu_{1}}, T^{\mu_{2}}, \ldots, T^{\mu_{m}})$ under the stacking map by determining its $(i,j)$-th entry, as follows:  for a fixed $1\le i \le m$ and $1\le j\le \mu^\Sigma_{i}$ let $k$ be the maximal index such that $\mu_{1,i} + \mu_{2,i} + \ldots + \mu_{k-1,i} < j$ (once again $\mu_{i,j} = 0$ for $j > m_i$).  Put
\begin{equation} \label{eq:stack}
\displaystyle \tilde{j}:= j -(\mu_{1,i} + \mu_{2,i} + \ldots + \mu_{k-1,i} ).
\end{equation}
Then define 
\begin{equation} 
t_{i,j}^{\Sigma}:=t_{i,\tilde{j}}^{\mu_{k}} + \sum_{l=1}^{k-1} \lambda_{l}
\end{equation}
to be the $(i,j)$-th entry of the tableau $T^{\Sigma}$, and we set
$$
\stk(T^{\mu_{1}}, T^{\mu_{2}}, \ldots, T^{\mu_{m}}  ) := T^\Sigma.
$$
Since each $T^{\mu_{i}}$ is standard, it follows that $T^{\Sigma}$ is standard of shape $\mu^{\Sigma}$. 

\begin{example}
Let $N=15$ and consider the partition $\lambda=(6,5,4)$. Define three further partitions of the parts of $\lambda$ as follows: $\mu_{1}=(3,3), \mu_{2}=(2,2,1)$ and $\mu_{3}=(1^4)$; and now consider three standard Young tableaux of shape $\mu_{i}$ for $i=1,2,3$. 
$$
\left( \young(134,256), \young(13,25,4), \young(1,2,3,4)\right)
$$
Then the new partition $\mu^{\Sigma}$ is $(6,6,2,1)$ and the image under the stacking map will be
$$
T^{\Sigma}=\young(13479\yb,2568\ya\yc,\y\yd,\ye)
$$
We explicitly confirm the $(3,2)$ entry of $T^{\Sigma}$. Set $i=3$ and $j=2$. Then the maximal $k$ such that $\mu_{1,3}+\mu_{2,3} + \ldots + \mu_{k-1,3} < 2$ is $k=3$ (where we note that $\mu_{1,3}=0)$. Hence $\tilde{j}=2-(\mu_{1,3}+\mu_{2,3})=2-(0+1)=1$ and so
$$
t_{3,2}^{\Sigma}=t_{3,1}^{\mu_{3}} + \sum_{l=1}^{2} \lambda_l = 3+(6+5)=14.
$$

Let us also calculate the $(3,1)$ entry of $T^{\Sigma}$. Set $i=3$ and $j=1$. Then the maximal $k$ such that $\mu_{1,3}+\mu_{2,3} + \ldots + \mu_{k-1,3} < 1$ is $k=2$, since $\mu_{1,3}=0$ and $\mu_{2,3}=1$. Hence $\tilde{j}=1 - (\mu_{1,3})=1-0=1$ and so:
$$
t_{3,1}^{\Sigma}=t_{3,1}^{\mu_{2}}+\sum_{l=1}^{1}\lambda_{l}=4+6=10
$$ 
\end{example}

\begin{proposition}
The tableau $T^{\Sigma}=\stk(T^{\mu_{1}}, T^{\mu_{2}}, \ldots, T^{\mu_{m}})$ is standard.
\end{proposition}

\begin{proof}
We first show that $T^{\Sigma}$ is increasing along columns. Fix $i,j$ and $j'$ in the appropriate range and suppose that $j < j'$. Let $k$ be maximal such that $\mu_{1,i}+ \ldots + \mu_{k-1,i} < j$ and let $k'$ be maximal such that $\mu_{1,i}+ \ldots + \mu_{k'-1,i} < j'$. Since $j < j'$, it follows that $k \leq k'$. Suppose that $k=k'$. Put $\displaystyle \tilde{j}=j-\sum_{l=1}^{k-1} \lambda_{l}$ and $\tilde{j'}=j-\sum_{l=1}^{k'-1} \lambda_{l}$. It is clear that $\tilde{j} < \tilde{j'}$ and since $T^{\mu_{k}}$ is standard, we have $t_{i,\tilde{j}}^{\mu_{k}} < t_{i,\tilde{j'}}^{\mu_{k}}$ and so $t_{i,j}^{\Sigma} < t_{i,j'}^{\Sigma}$. If $k<k'$, then since $1\leq t_{i,\tilde{j}}^{\mu_{k}} \leq \lambda_{k}$ we have:
$$
t_{i,j}^{\Sigma} = t_{i,\tilde{j}}^{\mu_{k}} + \sum_{l=1}^{k-1}\lambda_{l} < \sum_{l=1}^{k'-1}\lambda_{l} <  t_{i,\tilde{j'}}^{\mu_{k'}} + \sum_{l=1}^{k-1}\lambda_{l} = t_{i,j'}^{\Sigma}
$$
Hence $T^{\Sigma}$ is increasing along columns. \vspace{5pt}

We now show that $T^{\Sigma}$ is increasing down rows. Now fix $i, i'$ and  $j$ in the appropriate range and suppose that $i < i'$. Then $k \leq k'$. If $k=k'$, then since $T^{\mu_{k}}$ is a standard tableau, it follows that $t_{i,j}^{\mu_{k}} < t_{i',j}^{\mu_{k}}$. If $k < k'$ then again since $1 \leq t_{i,j}^{\mu_{k}} \leq \lambda_{k}$, we have:
$$
t_{i,j}^{\Sigma} = t_{i,\tilde{j}}^{\mu_{k}} + \sum_{l=1}^{k-1}\lambda_{l} < \sum_{l=1}^{k'-1}\lambda_{l} <  t_{i',\tilde{j}}^{\mu_{k'}} + \sum_{l=1}^{k-1}\lambda_{l} = t_{i',j}^{\Sigma}
$$
and so $T^{\Sigma}$ is increasing down rows. Hence $T^{\Sigma}$ is a standard tableau. 
\end{proof}

\section{Lie algebras and Springer fibres}

\subsection{Lie algebras of reductive algebraic groups}

Fix a $p \ge 0$ either zero or prime. Pick once and for all an algebraically closed field $\k$ of characteristic $p$. All vector spaces, algebras and algebraic varieties will be defined over $\k$. If $X$ is any variety over $\k$ then we write $\Comp(X)$ for the set of irreducible components. Let $G$ be a reductive algebraic group over $\k$ and assume the standard hypotheses (\cite[§2.9]{JaNO}) so that, in particular, $p$ is a good prime for the root system of $G$.\\

Write $\g = \Lie(G)$. When $p > 0$ we write $x \mapsto x^{[p]}$ for the natural $G$-equivariant restricted structure on $\g$. The definition of a nilpotent element of $\g$ depends on whether or not $p = 0$. For $p = 0$ a {\it nilpotent element} $e\in \g$ is one which acts nilpotently on every finite dimensional representation. For $p > 0$ a {\it nilpotent element} is one satisfying $e^{[p]^i} = 0$ for $i \gg 0$. For example when $G = \GL_N$ nilpotent elements are just the matrices which act nilpotently on the natural representation $\k^N$ of $\g$ (this description is independent of $p$). Write $\cN(\g) \subseteq \g$ for the {\it nilpotent cone}, defined to be the closed algebraic subvariety consisting of nilpotent elements.\\

When $x\in G$ we write $G^x$ and $\g^x$ for the stabliser and centraliser respectively. By \cite[§2.9]{JaNO} we have
\begin{eqnarray}
\label{e:centraliservsLiealg}
\Lie(G^x) = \g^x
\end{eqnarray}
for all $x\in \g$.\\ 

As usual $\B$ denotes the projective algebraic variety consisting of all Borel subalgebras of $\g$, the flag variety of $\g$. If we pick a Borel subgroup $B = T \ltimes N$, maximal torus $T$ and unipotent radical $N$, and Lie algebra $\b = \t \oplus \n$ then we may identify $G/B \isoto \B$ via the morphism of varieties $gB \mapsto g\cdot\b$ (see \cite[§10]{JaNO} for example). If $e \in \g$ is nilpotent then the {\it Springer fibre of $e$} is the closed subvariety
\begin{eqnarray}
\B_e := \{ \b \in \B \mid e\in \b\}
\end{eqnarray}
Identifying with $G/B$ this is equal to $\{gB \in G/B \mid e\in \Lie(g\cdot \b)\}$.
This is equal to the fibre over $e$ of Springer's resolution $\widetilde \cN(\g) \onto \cN(\g)$ of the nilpotent cone. The dimensions of Springer fibres are conveniently described as follows.
\begin{thm} \cite[Theorem~10.11]{JaNO}
\label{L:Springerdimension}
$\B_e$ is of pure dimension, and every irreducible component has dimension
$$\dim G/B - \frac{1}{2} \dim (G\cdot e) = \frac{1}{2}\codim_{\cN(\g)} (G\cdot e).$$
\end{thm}
The goal of this paper is to study the combinatorics of $\Comp(\B_e)$ for $G = \GL_n$ using the theory of induced nilpotent orbits.

\subsection{Lusztig--Spaltenstein induction for Springer fibres}

The theory of (Lusztig--Spaltenstein) induced unipotent classes was first introduced in \cite{LS} over $\C$. In this section we recap properties of induced nilpotent orbits for the Lie algebra of a reductive group, under the standard hypothesis, following \cite{PS}. We go on to explain how induction of orbits gives rise to a closed morphism of relative dimension zero between Springer fibres (Proposition~\ref{prop:LSmap}); this result was presumably well-known to experts. \\

A {\it Levi subalgebra of $\g$} is the Lie algebra of a Levi factor of a parabolic subgroup of $G$. Let $\g_0 \subseteq \g$ be a Levi subalgebra and let $\p$ be a parabolic subalgebra admitting $\g_0$ as a Levi factor. Write $\p = \g_0 \oplus \n$ for a Levi decomposition of $\p$, where $\n$ is the nilradical of $\p$. \\

If $\O_0$ is any nilpotent orbit in $\g_0$ then it is easily seen that $\Ad(G)(\O_0 + \n)$ contains a unique dense $G$-orbit which we denote $\Ind_{\g_0}^\g(\O_0)$, and call the {\it induced orbit from $(\g_0, \O_0)$}. As the notation suggests it only depends on the $G$-orbit of the pair $(\g_0, \O_0)$ and not on the choice of parabolic $\p$ containing $\g_0$ as a Levi factor. Furthermore induction satisfies the following two important properties:
\begin{lem}
\label{L:inductionproperties}
\begin{enumerate}
\setlength{\itemsep}{4pt}
\item (Transitivity) If $\g_0 \subseteq \g_1 \subseteq \g$ are Levi subalgebras and $\O_0 \subseteq \g_0$ is a nilpotent orbit then $\Ind_{\g}^\g(\O_0) = \Ind_{\g_0}^{\g_1} \Ind_{\g_1}^\g(\O_0).$
\item (Preservation of codimension) With $\O_0\subseteq \g_0$, \ $e_0 \in \g_0$ and $e\in \Ind_{\g_0}^\g(\O_0)$ we have $\dim \g_0^{e_0} = \dim \g^e$.
\item $\O_\p$ is a single $P$-orbit.
\end{enumerate}
\end{lem}
\begin{proof}
The first two parts were first observed under the standard hypotheses in \cite[§2.5]{PS}. The third part was proven by Lusztig--Spaltenstein \cite[Theorem~1.3(c)]{LS} in the setting of complex algebraic groups, and the same proof works in our setting, applying \cite[Theorem~2.6(iv)]{PrKR}.
\end{proof}

Now let $\p\subseteq \g$ be a parabolic subalgebra and $\g_0$ a Levi factor. Pick a nilpotent orbit $\O_0\subseteq \g_0$ and write $\O_\p = \Ind_{\g_0}^\g(\O_0) \cap \p$. Pick an element $e \in \O_\p$ and write $e = e_0 + e_1$ for the decomposition of $e$ across $\p = \g_0 \oplus \n$. In this paper we study the following map
\begin{eqnarray}
\label{e:LSdefn}
\begin{array}{rcl}
\LS & : & \B^0 \longrightarrow \B,\\
& & \b_0 \longmapsto \b_0 \oplus \n.
\end{array}
\end{eqnarray}
We call this the {\it Lusztig--Spaltenstein morphism of Springer fibres}. Some basic properties are listed here.
\begin{proposition}
\label{P:LSwelldefined}
\begin{enumerate}
\setlength{\itemsep}{4pt}
\item $\LS$ is well-defined $G_0$-equivariant morphism of algebraic varieties.
\item $\LS$ restricts to a map $\B_{e_0}^0 \to \B_e$ which is closed and of relative dimension zero.
\end{enumerate}
\end{proposition}
\begin{proof}
Let $B \subseteq P$ be any Borel subgroup and pick a torus $T \subseteq B$. Let $\Phi \subseteq X^*(T)$ be the corresponding root system and $\Delta \subseteq \Phi$ the set of simple roots corresponding to $B$. Thanks to the classification of parabolic subgroups and their Levi factors (see \cite[Theorems~30.1\&~30.2]{Hum}, for example) we can choose $T$ so that $P$ is a standard parabolic with respect to $\Delta$. This means that after choosing a Levi factor $G_0$ containing $T$, and writing $\Phi_0 \subseteq \Phi$ for the root system of $G_0$, we have that $\Delta_0 := \Phi_0 \cap \Delta$ is a set of simple roots in $\Phi_0$.

We have implicitly chosen positive roots $\Phi_0^+,  \Phi^+$. Write $\Phi^+_1 = \Phi^+ \setminus \Phi^+_0$. Pick a set $\{x_\alpha \mid \alpha \in \g\} \subseteq \Phi$ of root vectors. Since $P$ is standard with respect to $\Delta$ we have $\p = \ub + \g_0$ and so $\n = \sum_{\alpha \in \Phi^+_1} \k x_\alpha$.

Since $G_0$ preserves $\n$ it follows from the definition of $\LS$ that for any Borel $\b_0 \subseteq \g_0$ we have $\LS(g\cdot \b_0) = g\cdot \LS(\b_0)$ for all $g\in G_0$. So $\LS$ is a $G_0$-equivariant map sending vector spaces of $\g_0$ to subspaces of $\g$.

Now let $\b_0 \subseteq \g_0$ be the unique Borel subalgebra containing $\t$ and $\{x_\alpha \mid \alpha \in \Phi_0^+\}$ for $\alpha \in \Phi_0^+$. By construction $\LS(\b_0) = \b$. Since $\B^0 = G_0 \cdot \b_0$ the $G_0$-equivariance implies that $\LS(\B^0)\subseteq \B$. So $\LS$ is well-defined.

We now check that $\LS$ is a morphism. Let $B_0$ be the Borel subgroup of $G_0$ with Lie algebra $\b_0$ chosen in the previous paragraph, and $B_0^- \subseteq G_0$ be the Borel subgroup of $G_0$ opposite to $B_0$. Write $U_0^-$ for the unipotent radical. Recall that the big cell $\Omega_0 \subseteq \B_0$ is the image of $U_0^-$ under the map $G_0 \to \B^0$ given by $g\mapsto g\cdot \b_0$ (see \cite[28.5]{Hum} for example), and that the map $U_0^- \to \B^0$ is injective. Since $\LS$ is $G_0$-equivariant and $\B^0$ has an affine cover by $G_0$-translates of $\Omega_0$ it will suffice to show that $\LS|_{\Omega_0}$ is a morphism. Writing $U^-$ for the unipotent radical of the opposite Borel to $B$ we have a similar description of the big cell $\Omega\subseteq \B$. It follows from $B_0 \subseteq B$ that $\LS(\Omega_0) \subseteq \Omega$ and so it suffices to show that the pullback $\LS^* : \k[\Omega] \to \k[\Omega_0]$ is a homomorphism.

For $\alpha\in \Phi$ let $u_\alpha : \k \to G$ be the corresponding 1-parameter subgroup. For a fixed total order on $\Phi$ we have isomorphisms $\mathbb{A}^{\Phi^+} \isoto U^-$ and $\mathbb{A}^{\Phi^+_0} \isoto U^-_0$ given by $(t_{\alpha})_{\alpha \in \Phi^+} \mapsto \prod_{\alpha \in \Phi^+} u_{-\alpha}(t_{\alpha})$, and similar for $U^-_0$. Now after identifying through the isomorphism $\Omega \cong U^- \cong \mathbb{A}^{\Phi^+}$ and $\Omega_0 \cong  U_0^- \cong \mathbb{A}^{\Phi_0^+}$ we see that $\LS^*$ is just the projection homomorphism $\k[\mathbb{A}^{\Phi^+}] \onto \k[\mathbb{A}^{\Phi_0^+}]$ corresponding to the inclusion $\Phi_0^+ \subseteq \Phi_0$. Thus $\LS$ is a morphism.

The relative dimension of $\B_{e_0}^0 \to \B_e$ is zero thanks to \eqref{e:centraliservsLiealg}, Lemma~\ref{L:Springerdimension} and Lemma~\ref{L:inductionproperties}(2). Since $\LS$ is a morphism of projective (hence complete) varieties, it follows from \cite[Proposition~6.1]{Hum} that it is closed. This completes the proof. 
\end{proof}

\begin{cor}
\label{cor:inducedmaponcomponents}
The map on Springer fibres induces a map on the sets of components: for every $C \in \Comp(\B^0_{e_0})$ we have $\LS(C) \in \Comp(\B_e)$.
\end{cor}

\begin{rem}
One of the most surprising features of Lusztig--Spaltenstein induction of nilpotent orbits is that it depends on the conjugacy class of $\g_0$, not on the conjugacy class of $\p$. It would be interesting to know whether a similar independence statement can be formulated for \eqref{e:LSdefn}.
\end{rem} 

\section{General linear lie algebras}

For the rest of the paper keep $N > 0$ fixed and choose an algebraically closed field of any characteristic $p \ge 0$. Let $G = \GL_N(\k)$ and $\g = \Lie(G) = \gl_N(\k)$. Let $\kappa$ denote the trace form associated to the natural representation $V = \k^N$. We note that the standard hypotheses are satisfied for $G$.

\subsection{Nilpotent Orbits, Levi subalgebras and induction}
\label{ss:nilpotentorbitslevisandinduction}

The nilpotent elements of $\g$ are those which act nilpotently on the natural representation, and we denote the set of such elements $\cN(\g)$. If $e\in \cN(\g)$ then we can decompose $V$ non-uniquely into indecomposable $\k[e]$-modules $V = \bigoplus_{i=1}^n V_i$ which we refer to as (a choice of) {\it Jordan block spaces for $e$}. A {\it Jordan basis for $e$} is a basis $\{v_{i,j} \mid i=1,...,n, \ j = 1,\ldots, \dim V_i\}$ such that $V_i$ has basis $\{v_{i,j} \mid j=1,\ldots, \dim(V_i)\}$
$$ev_{i,j} = \left\{\begin{array}{cc} v_{i, j-1} & \text{ if } j > 1 \\ 0 & \text{ if } j = 1\end{array} \right.\vspace{6pt}$$

The $G$-orbits on $\cN(\g)$ are classified by partitions: for each $\lambda = (\lambda_1,...,\lambda_n) \vdash N$ we let $\O_\lambda$ denote the $G$-orbit consisting of elements with Jordan block spaces of dimension $\lambda_1, \lambda_2, ...,\lambda_n$.\\

The Levi subalgebras of $\g$ are also classified by partitions of $N$: for $\lambda \vdash N$ choose any vector space decomposition $V = \bigoplus_{i=1}^n V_i$ where $\dim V_i = \lambda_i$, let $\g_\lambda \cong \bigoplus \gl_{\lambda_i}$ be the subalgebra of $\g$ which preserves each $V_i$. This defines a bijection from partitions of $N$ to conjugacy classes of Levi subalgebras (see \cite[§7.2]{CM} for example).\\

Now we are in a position to describe Lusztig--Spaltenstein induction of nilpotent orbits. Let $\lambda \vdash N$ and let $\g_\lambda$ be a choice of Levi subalgebra, as described above. Suppose that $G_\lambda \subseteq G$ is a Levi subgroup with $\g_\lambda = \Lie(G_\lambda)$. The basic result for describing induced nilpotent $G$-orbits is due to Kraft and Ozeki--Wakimoto, independently.
\begin{lem} \cite[7.2.3]{CM}
\label{L:inducedfromzero}
If $\O_0$ is the zero orbit in $\g_\lambda$ then the partition associated to $\Ind_{\g_\lambda}^\g(\cO_0)$ is the transpose $\lambda^\top$.
\end{lem}

We need to upgrade this result to describe induction from non-zero orbits. For each $i=1,...,n$ choose a partition $\mu_i = (\mu_{i,1},...,\mu_{i,m_i}) \vdash \lambda_i$. By the above remarks there is a nilpotent $G_0$-orbit $\O_\mu \subseteq \g_0$ such that the projection to the factor $\gl_{\lambda_i}$ is the $\GL_{\lambda_i}$-orbit classified by partition $\mu_i$.\\

Recall the definition of $\mu^\Sigma$ from \eqref{e:musigmadefn}.
\begin{cor}
\label{cor:partitionsigma}
The partition of $\Ind_{\g_\lambda}^\g(\cO_\mu)$ is $\mu^\Sigma$.
\end{cor}
\begin{proof}
Writing $\g_\lambda = \bigoplus_{i=1}^n \gl_{\lambda_i}$ we pick a Levi subalgebra $\g_{\mu_i^\top}$ of $\gl_{\lambda_i}$ which has partition of type $\mu_i^\top$. Write $\g_{\mu^\top} = \bigoplus_{i=1}^n \g_{\mu_i^\top} \subseteq \bigoplus_{i=1}^n \gl_{\lambda_i}$ and note that $\g_{\mu^\top}$ is a Levi subalgebra of $\g$. Write $\O_0$ for the zero orbit in $\g_{\mu^\top}$. Then according to Lemma~\ref{L:inducedfromzero} $\O_\mu = \Ind_{\g_{\mu^\top}}^{\g_\lambda}(\O_0)$. Now by the transitivity of induction (Lemma~\ref{L:inductionproperties}) we see that $\Ind^\g_{\g_\lambda}(\O_\mu) = \Ind_{\g_{\mu^\top}}^\g(\O_0)$. 

Now observe that if we concatenate $\mu_1^\top, \mu_2^\top,...,\mu_n^\top$ and reorder to make a partition $\nu \vdash N$ then $\nu^\top = \mu^\Sigma$. The proof concludes by applying Lemma~\ref{L:inducedfromzero} once more, to show that the partition of $\Ind_{\g_{\mu^\top}}^\g(\O_0)$ is $\nu^\top$.
\end{proof}

\subsection{A representative for the induced orbit}
\label{ss:representative}
In this Section we indicate a nice choice of representative for a nilpotent orbit $\O_0 \subseteq \g_0$ in a Levi subalgebra, and for the induced orbit $\O := \Ind_{\g_0}^\g(\O_0)$. This will be useful for illustrating some of our arguments below.

Fix a composition $\lambda = (\lambda_1, \lambda_2, \cdots, \lambda_n)$ of $N$. For each $i =1,...,n$ we pick a partition $\mu_i = (\mu_{i,1},...,\mu_{i, m_i})$ of $\lambda_i$. Now from this data we define a set
\begin{eqnarray}
\I = \{(i,j,k) \mid i=1,...,n, \ j=1,...,m_i, \ k=1,...,\mu_{i,j}\}
\end{eqnarray}
and we let $V$ be the $N$-dimensional complex vector space with basis $\{v_{i,j,k} \mid (i,j,k) \in \I\}$. We identify $V$ with the natural representation of the general linear Lie algebra $\g := \gl_N$, and so $\g$ admits a basis
\begin{eqnarray}
& &\{e_{i_1,j_1,k_1; i_2,j_2,k_2} \mid (i_1,j_1,k_1), (i_2,j_2,k_2) \in \I\},\\
& & e_{i_1,j_1,k_1; i_2,j_2,k_2}v_{i_3,j_3,k_3} = \delta_{i_2, i_3} \delta_{j_2, j_3} \delta_{k_2, k_3} v_{i_1, j_1, k_1}.
\end{eqnarray}

Now we define the Levi subalgebra $\g_0 \subseteq \g$ to be the subalgebra spanned by elements $\{e_{i, j_1, k_1; i, j_2, k_2} \mid i=1,...,n, \ \ (i, j_1, k_1), (i, j_2, k_2) \in \I\}$. This algebra is isomorphic to $\gl_{\lambda_1}\oplus \cdots \oplus \gl_{\lambda_m}$.

There is a corresponding decomposition of $V$: for $i=1,...,n$ fixed we let $V_i$ be the subspace spanned by $v_{i,j,k}$, allowing $j,k$ to vary. Then we have $V = \bigoplus_i V_i$, and $V_i$ identifies with the natural representation of $\gl_{\lambda_i}$

There is a parabolic subalgebra $\p \subseteq \g$  admitting $\g_0$ as a Levi factor, which is spanned by elements $e_{i_1, j_1, k_1; i_2, j_2, k_2}$ where $i_2 \ge i_1$. The nilradical $\n \subseteq \p$ consists of elements with $i_2 > i_1$.

Now we define some nilpotent elements. We let
\begin{eqnarray}
\label{e:e0defn}
e_0 := \sum_{i=1}^m \sum_{j=1}^{m_i} \sum_{k=1}^{\mu_{i,j} - 1} e_{i,j,k; i,j,k+1} \in \g_0.
\end{eqnarray}
This is a nilpotent element of $\g_0$. When restricted to the subspace $V_i$ the associated nilpotent operator has partition $\mu_i$. Now we define an element $e_1 \in \n$. Let $d_j$ be the number of indexes $1\le i \le n$ such that $m_i \ge j$ and let $\{i_1^j,...,i_{d_j}^j\} \subseteq \{1,...,n\}$ be the indexes satisfying $m_{i_k} \ge j$ for $k = 1,...,d_j$. Now we let
\begin{equation}
\label{e:e1defn}
e_1 := \sum_{j > 0} \sum_{k=1}^{d_j-1} e_{i_{k}^j, j, \mu_{i_{k}^j, j}; i_{k+1}^j, j, 1}
\end{equation}
and define
\begin{eqnarray}
e := e_0 + e_1.
\end{eqnarray}

The elements $e_0, e_1$ are easily understood pictorially, identifying the elements of the Jordan basis with the boxes in a tuple of Young diagrams. We illustrate this with an example.

\begin{example} 
\label{ex:representativediagram}
Let $\lambda=(7, 5, 12)$ be a composition of $24$ and let $$(\mu_1, \mu_2, \mu_3)=((4,2,1), (3,2), (3^2,2^2,1^2)).$$ In the following diagram, the elements of the Jordan basis $v_{i,j,l}$ are identified with the boxes of the Young diagrams in the obvious manner. The action of $e_0$ is illustrated in black, and $e_1$ in blue, as follows:

\begin{center}
\begin{tikzpicture}[scale=1]
\node at (0,0) {$\yng(4,2,1)$};
\node at (3,0.25) {$\yng(3,2)$};
\node at (6,-0.7) {$\yng(3,3,2,2,1,1)$};

\draw[->] (0.7,0.8) to[bend right]node[above]{\tiny $e_0$} (0.3,0.8);
\draw[->] (0.2,0.8) to[bend right] (-0.2,0.8);
\draw[->] (-0.3,0.8) to[bend right] (-0.7,0.8);
\draw[->] (-0.8,0.8) to[bend right] (-1.2,0.8) node[ above]{\tiny $0$};

\draw[->] (3.5,0.8) to[bend right]node[above]{\tiny $e_0$}  (3.1,0.8);
\draw[->] (3,0.8) to[bend right] (2.6,0.8);
\draw[->] (2.5,0.8) to[bend right] (2.1,0.8)node[ above]{\tiny $0$};

\draw[->] (6.5,0.8) to[bend right]node[above]{\tiny $e_0$}  (6.1,0.8);
\draw[->] (6,0.8) to[bend right] (5.4,0.8);
\draw[->] (5.3,0.8) to[bend right] (4.9,0.8) node[ above]{\tiny $0$};

\draw[blue, ->] (2,0.5) to node[below] {\tiny $e_{1}$} (1.3,0.5);
\draw[blue, ->] (2,0) to  (0.3,0);

\draw[blue, ->] (5,0.5) to node[below] {\tiny $e_{1}$} (4.3,0.5);
\draw[blue, ->] (5,0) to  (3.3,0);
\draw[blue, ->] (5,-0.45) to  (-0.3,-0.45);

\draw[blue, ->] (5,-.9) to  (4.5,-.9)node[left]{\tiny $0$};
\draw[blue, ->] (5,-1.4) to  (4.5,-1.4)node[left]{\tiny $0$};
\draw[blue, ->] (5,-1.9) to  (4.5,-1.9) node[left]{\tiny $0$};
\end{tikzpicture}\\
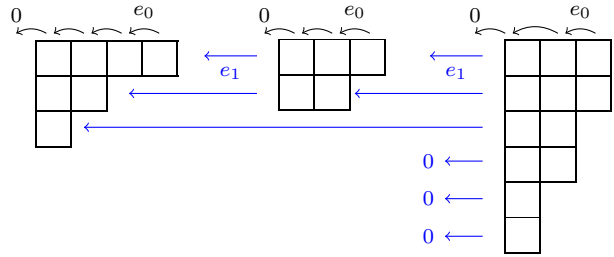
{\sf Diagram 1: The action of $e_0$ and $e_1$ on the natural representation.}
\end{center}

Hence we may think of $e_0$ and $e_1$ as acting on the boxes where:
\begin{itemize}
\item  $e_0$ moves a box in a given Young diagram to the box immediately to its left, unless it is in the first column of a Young diagram, in which case it sends it to $0$; and
\item $e_1$ sends a box in a given Young diagram to $0$ unless it is in the first column and there exists a box in the Young diagram immediately to to the left, 
in which case it sends it to the box to its immediate left. 
\end{itemize}

At this point, it is also worth illustrating how $e:=e_0 + e_1$ acts on the $N$-dimensional space $V$, via the `stacked' Young diagram. We picture the action as follows:

\begin{center}
\begin{tikzpicture}
\draw[->,black] (-2,1.5) to[bend right]node[above]{\tiny $e$}  (-2.4,1.5);
\foreach \i [evaluate={\j=\i-0.4; }] in {2,1.5,1,...,-1.5}
\draw[->,black] (\i,1.5) to[bend right] (\j,1.5);
\end{tikzpicture}\vspace{2pt}\\
\hspace{15pt}{\tiny 
 \ydiagram[*(darkgray)]
  {4,2,1}
  *[*(gray)]{7,4}
 * [*(white)]{10,7,3,2,1,1}}\\
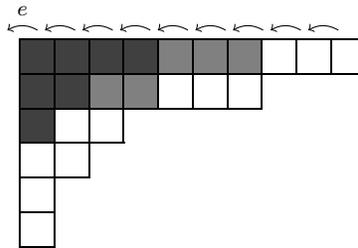
 {\sf Diagram 2: The action of $e$ on the natural representation.}
\end{center}
Hence we may see that for each $ 1 \leq j \leq 10 =\sum_{k=1}^{3}\mu_{k,1}$, we may picture $\ker(e^j)$ as the span of the boxes in the first $j$ columns of the stacked Young diagram. The colours of the boxes will be useful for illustrating certain arguments later on.
\end{example}

\begin{lem}
$e\in \O_\p$.
\end{lem}
\begin{proof}
For $j=1,...,\max m_i$ fixed the basis vectors $v_{i,j,l}$ (allowing $i,l$ to vary) span a single Jordan block for $e$ of size $\sum_{i=1}^n \mu_{i,j} = \mu^\Sigma_j$. It follows that the $G$-orbit of $e$ has partition $\mu^\Sigma$ and the Lemma follows, thanks to Corollary~\ref{cor:partitionsigma}
\end{proof}

\subsection{Flags and Spaltenstein's description of the components of Springer fibres}
Let $e$ be a nilpotent element in $\gl_{N}$ of Jordan type $\lambda$ and let $\B_e$ be the corresponding Springer fibre over $e$. We recall that the points of $\B$ are described in terms of full flags of $V = \k^N$: a flag $F_\bullet = (0\subsetneq F_1\subsetneq \cdots \subsetneq F_{N-1} \subsetneq \k^N)$ corresponds to the Borel subalgebra consisting of elements of $\g$ preserving each $F_i$. Thus we may explicitly describe the Springer fibre in terms of flags as follows
$$
\B_e = \{ 0 \subset F_1  \subset \cdots \subset F_{N-1} \subset \C^N \, | \, \dim(F_i)=i, \ e(F_i) \subseteq F_{i-1} \} .
$$

For a fixed flag $F_{\bullet} \in \B_{e}$, by considering the Jordan type of $e$ restricted to the subspaces $F_i$, we obtain a natural map from the fibre to sequences of partitions: 
\begin{eqnarray}
\label{e:Spalmapdefn}
\begin{array}{rcl}
\Phi & : & \B_e  \longrightarrow \mathcal{P}_1 \times \mathcal{P}_2 \times \ldots \times \mathcal{P}_{N} \vspace{4pt} \\
	& & F_{\bullet}  \longmapsto (\Type(e_{|F_1}), \ldots, \Type(e_{|F_{n}})) 
	\end{array}
\end{eqnarray}

We identify partitions of $N$ with their corresponding Young diagrams. 

\begin{lemma} \label{lemma:spaltenstein}
Let $e$ be a nilpotent of Jordan type $\lambda$ and let $H$ be an $e$-stable hyperplane of $\C^n$. Then $\Type(e|_{H})$ is obtained by removing the last box from the $j$-column of $\lambda$ where $j$ is maximal such that $H \supseteq \ker(e^{j-1})$. 
\end{lemma}

\begin{proof}
We may choose a Jordan basis for the nilpotent $e$ on $V$ such that $H$ is spanned by all but one of those Jordan basis vectors. The result follows. 
\end{proof}

By a {\it nested sequence of partitions} we mean a sequence of partitions, or Young diagrams, for each $1 \leq k \leq n$ such that each Young diagram is obtained from the previous by adding a box to an available row, where available means that the diagram obtained by adding the box is indeed a Young diagram. \\

Lemma \ref{lemma:spaltenstein} shows that the image of $\Phi$ is the set of nested sequences of partitions which are in bijection with $\Std(\lambda)$. For instance, the standard tableau 
\begin{eqnarray}
\label{e:atab}
\young(134,25)
\end{eqnarray}
 corresponds to the nested sequence:
$$
\left(\varnothing,  \yng(1), \yng(1,1), \yng(2,1), \yng(3,1), \yng(3,2) \right)
$$
Hence we may identify the image of $\Phi$ with a subset of $\Std(\lambda)$ and a simple induction argument shows that $\Phi$ is a surjection onto $\Std(\lambda)$ and we have the following due to Spaltenstein:
\begin{theorem}\cite{Spa}
Let $e$ be nilpotent of Jordan type $\lambda$. Then:
\begin{enumerate}
\setlength{\itemsep}{4pt}
\item there is a bijection between $\Comp(\B_{e})$ and $\Std(\lambda)$ induced by $\Phi$ where for any $\sigma \in \Std(\lambda)$, $\overline{\Phi^{-1}(\sigma)}$ is an irreducible component. 
\item  (Cf. Theorem~\ref{L:Springerdimension}). $\B_{e}$ is of pure dimension and for all $\sigma \in \Std(\lambda)$,  $$\dim(\overline{\Phi^{-1}(\sigma)})=\frac{1}{2}\sum_{i \geq 1} \lambda_i^\top(\lambda^\top_i-1).$$
\end{enumerate}
\end{theorem}

For $\sigma \in \Std(\lambda)$, put $X_{\sigma}:=\Phi^{-1}(\sigma)$. Using the ordering on standard tableaux given in Section~\ref{subsection:ordering} we have the following important property of closures.

\begin{lem}
\label{lem:Xsigmaintersectsclosure}
We have $X_\sigma \cap \overline{X}_\tau = \emptyset$ for $\tau > \sigma$.
\end{lem}
\begin{proof}
By \cite[Proposition, (a)]{Spa} we know that $X_\sigma$ is locally closed and $\bigcup_{\tau > \sigma} X_\sigma$ is closed in $\B_e$. Therefore $\overline X_\tau \subseteq \bigcup_{\tau' \ge \tau} X_{\tau'}$ and the lemma follows from the fact that the fibres of Spaltenstein's map are disjoint.
\end{proof}

As we observed in Section~\ref{ss:nilpotentorbitslevisandinduction} the $G_0$-orbit $\O_0$ is determined by partitions $\mu_1,...,\mu_n$ such that $\mu_i \vdash \lambda_i$. Note that Spaltenstein's map gives a map $\Phi_0 : \B_{e_0}^0 \to \prod_{i=1}^n \Std(\mu_i)$, and the fibre over a tuple $(\sigma^{(1)},...,\sigma^{(n)})$ of standard tableaux is denoted $X_{\sigma^{(1)},...,\sigma^{(n)}}$. 

The next Lemma, combined with Lemma~\ref{L:inductionproperties}(3), shows that in order to understand $\LS$ for any element $e\in \O_\p$ it suffices to understand the morphism for a single element. Thus our representative chosen above may be seen as a typical element for our purposes. It follows directly from the definition of $\LS$ and $\Phi$, see  \eqref{e:LSdefn} and \eqref{e:Spalmapdefn}. 
\begin{proposition}
\label{L:reductionlemma}
Suppose that $\p = \Lie(P)$, $g \in P$ and let $g = g_0 u \in G_0U$ where $U = \Rad(P)$. Then we have a commutative diagrams
\begin{center}
\begin{tikzpicture}[node distance=2cm, auto]
 \node (A) {$\B^0_{e_0}$};
 \node (B) [below of= A] {$\B^0_{g_0\cdot e_0}$};
 \node (C) [right of= A] {$\B_e$};
  \node (D) [right of= B]{$\B_{g\cdot e}$};
 \draw[->] (A) to node [swap]{$g_0$} (B);
 \draw[->] (A) to node {$\LS$} (C);
 \draw[->] (B) to node [swap]{$\LS$} (D);
 \draw[->] (C) to node {$g$} (D);
  \node (E) [right of= C] {$\B_{e}$};
 \node (F) [below of= E] {$ $};
 \node (G) [right of= F] {$\Std(\lambda)$};
  \node (H) [right of= E] {$ $};
   \node (I) [right of= H] {$\B_{g\cdot e}$};
    \draw[->] (E) to node [swap]{$\Phi$} (G);
    \draw[->] (E) to node {$g$} (I);
    \draw[->] (I) to node {$\Phi$} (G);
      \node (J) [right of= I] {$\B_{e_0}$};
 \node (K) [below of= J] {$ $};
 \node (L) [right of= K] {$\prod_{i=1}^n \Std(\mu_i)$};
  \node (M) [right of= J] {$ $};
   \node (N) [right of= M] {$\B_{g_0\cdot e_0}$};
    \draw[->] (J) to node [swap]{$\Phi_0$} (L);
    \draw[->] (J) to node {$g_0$} (N);
    \draw[->] (N) to node {$\Phi_0$} (L);
\end{tikzpicture}
\end{center}
\end{proposition}

\subsection{The Lusztig--Spaltenstein map on fibres}
\label{ss:LSonfibres}

Let $\p$ be a parabolic subalgebra of $\g$ and pick a Levi factor $\g_0 \subseteq \p$. Let $\lambda = (\lambda_1,...,\lambda_n)$ denote the ranks of the general linear factors of $\g_0$. Observe that the decomposition $\g_0 \cong \bigoplus_{j=1}^n \gl_{\lambda_j}$ gives a decomposition $\C^N = \bigoplus_{j=1}^N V_j$ where any $V_j$ is stable under the action of every $\gl_{\lambda_k}$ and the restriction of $\gl_{\lambda_j}$ to endomorphisms of $V_j$ is an isomorphism. We order the spaces $V_j$ in such a way that $\p$ maps $V_j$ to $\bigoplus_{k=1}^{j} V_j$, and we note that $\lambda$ is a composition of $N$, not a partition in general.\\

We pick a nilpotent orbit $\O_0 \subseteq \g_0$ and choose an element $e\in \Ind^\g_{\g_0}(\O_0) \cap \p$. Let $e = e_0 + e_1$ be the decomposition of $e$ across the direct sum $\p = \g_0 \oplus \Rad(\p)$. Now we can consider the map $$\LS : \B_{e_0}^0 \longrightarrow \B_e.$$

By our choices above we note that setting $W_j = \bigoplus_{k=1}^j V_k$ determines a partial flag of $\k^N$, and $\p$ is described by $\p = \{x\in \g \mid xW_j \subseteq W_j \text{ for all } j\}$. In the following we regard the elements of $\B$ as full flags of $\C^N$, and similarly we regard an element of $\B^0$ as a tuple $(F^{(1)},...,F^{(n)})$ where each $F^{(j)}$ is a flag of $V_j$. In order to describe $\LS$ in this language we must rephrase the map in terms of flags.
\begin{lem} We have
\label{L:LSonflags}
$$
\LS(F_{\bullet}^{(1)}, \ldots, F_{\bullet}^{(n)}): ={\oF}_{\bullet} =(0 \subset \oF_{1} \subset  \oF_{2} \subset \cdots \subset  \oF_{N-1} \subset \k^{N})
$$
where 
$$
 \oF_{i}  = \begin{cases} F_{i}^{(1)} & \text{for} \; 1 \leq i \leq \lambda_{1} \\ W_{j-1} + F_{k}^{(j)} & \text{for}\; i > \lambda_{1} \; \text{where $j$ is maximal such that} \;  \displaystyle \sum_{l=1}^{j-1} \lambda_{l} < i\; \text{and} \; k=i - \sum_{l=1}^{j-1} \lambda_{l}. \end{cases} 
$$
\end{lem}

As we observed in Section~\ref{ss:nilpotentorbitslevisandinduction} the $G_0$-orbit $\O_0$ is determined by partitions $\mu_1,...,\mu_n$ such that $\mu_i \vdash \lambda_i$. Note that Spaltenstein's map gives a map $\Phi_0 : \B_{e_0}^0 \to \prod_{i=1}^n \Std(\mu_i)$, and the fibre over a tuple $(\sigma^{(1)},...,\sigma^{(n)})$ of standard tableaux is denoted $X_{\sigma^{(1)},...,\sigma^{(n)}}$. The next lemma follows directly from the definition of $\LS$ and $\Phi$, see  \eqref{e:LSdefn} and \eqref{e:Spalmapdefn}.

Our Main Theorem will follow fairly quickly from the next Proposition.
\begin{proposition} \label{prop:LSmap}
For each $i=1,...,n$ we choose a standard tableau $\sigma^{(i)}$ for $\mu_i$. Then
$$\LS(X_{\sigma^{(1)},...,\sigma^{(n)}}) \subseteq \bigcup_{\tau \ge \stk(\sigma^{(1)}, \ldots, \sigma^{(n)})} X_\tau.$$
\end{proposition}

\begin{proof}
Let $(F^{(1)},\cdots,F^{(n)}) \in \B^0_{e_0}$, and suppose that $(F^{(1)},\cdots,F^{(n)})\in \Phi_0^{-1}(\sigma^{(1)},...,\sigma^{(n)})$. Let $\LS(F^{(1)},...,F^{(n)}) = \oF_\bullet$, which is described explicitly in Lemma~\ref{L:LSonflags}. Suppose that $\oF_\bullet \in \Phi^{-1}(\tau)$ for some $\tau \in \Std(\lambda)$. Put $\sigma^{\Sigma}= \stk(\sigma^{(1)},...,\sigma^{(n)})$. In order to prove the proposition we must show that $\tau \ge \sigma^\Sigma$. If $\tau = \sigma$ then we are done, so we may assume that $\tau \ne \sigma$.\\

Recall that $\tau_i$ is the column of $\tau$ where $i$ occurs. To establish the Proposition, we need to show that there exists an index $i \in \{1,\ldots, N\}$ such that $\tau_{i} > \sigma_{i}^{\Sigma}$ and $\tau_{k} =  \sigma_{k}^{\Sigma}$ for $k > i$. Let $i$ be the maximal index such that $\tau_i \ne \sigma^\Sigma_i$. Now define $V' := \oF_i$ and we consider $\g' = \gl(V')$, $e' := e|_{V'} \in \g'$, $\g_0' := \g_0|_{V'}$, $G_0' := G_0|_{V'}$ and $\p' := \p|_{V'}$. Also write $\oF'_\bullet = (0\subseteq \oF_1 \subseteq \cdots \subseteq \oF_{i-1} \subseteq V')$ which is a full flag of $V'$. Write $\B'$ for the flag variety of $\g'$ and $\B'_{e'}$ for the Springer fibre, and $\Phi'$ for Spaltenstein's map. Note that $\oF'_\bullet \in \B'_{e'}$ and that $\Phi'(\oF_\bullet')$ consists of the boxes of $\tau$ with labels $1,...,i$.\\

Now let $j \in \{0,...,n-1\}$ be the largest index such that $\sum_{k=1}^j \lambda_k < i$, and set $d := i-\sum_{k=1}^{j-1}$. By Lemma~\ref{L:LSonflags} we know that $V_1,...,V_{j-1} \subseteq V'$ and $\tilde V_j := \oF_i \cap V_j$ satisfies $V_1 \oplus \cdots V_{j-1} \oplus \tilde V_{j}$. Therefore $\p'$ is the parabolic subalgebra of $\g'$ stabilising the partial flag $W'_\bullet$ of $V'$ given by $W'_k = W_k$ for $k=1,..,j-1$ and $W'_j = V'$, whilst $\g_0'$ is a Levi factor of $\p'$. It follows that the projection $e_0'$ of $e'$ across $\p' = \g_0' \oplus \nil(\p')$ is equal to the restriction $e_0|_{V'}$. The $G_0'$-orbit of $e_0'$ is determined by partitions $\mu_1',...,\mu_j'$ which are described as follows: $\mu_k = \mu_k'$ for $k=1,...,j-1$ and $\mu_j'$ is the shape of the tableau $\sigma^{(j)}{}'$ obtained from $\sigma^{(j)}$ by considering only the boxes labelled $1,2,...,d$.  \\

We claim that $e'$ lies in the nilpotent orbit in $\g'$ induced from $(\g_0', G_0'\cdot e_0')$. To see this write $\sigma^{\Sigma}{}'$ for the tableau formed from $\sigma^\Sigma$ by considering only the boxes with entries $1,...,i$. Similarly define $\tau'$. Using Corollary~\ref{cor:partitionsigma} and the remarks of the previous paragraph we see that the shape of $\sigma^{\Sigma}{}'$ is the Young diagram of the induced orbit $\Ind_{\g_0'}^{\g'}(G_0' \cdot e_0')$. On the other hand the shape of $\tau'$ is the Young diagram of the orbit of $e'$. Since we assumed that $\tau_j = \sigma^\Sigma_j$ for $j=i+1,...,N$ it follows that the shape of $\tau'$ is the shape of $\sigma^{\Sigma}{}'$, and this confirms the claim.\\

Let $F^{(1)}{}'_\bullet,...,F^{(j)}{}'_\bullet$ be the collection of flags of $V_1,...,V_{j-1}, \tilde V_j$ given by $F^{(k)}{}'_\bullet := F^{(k)}_\bullet$ for $k=1,..,j-1$ and $F^{(j)}{}'_\bullet := (0 \subseteq F^{(j)}_1\subseteq \cdots \subseteq F^{(j)}_{d-1} \subseteq \tilde V_j)$. By Lemma~\ref{L:LSonflags} the Lusztig--Spaltenstein map on fibres sends $(F^{(1)}{}'_\bullet,...,F^{(j)}{}'_\bullet)\in\B^0_{e_0'}{}'$ to $\oF'_\bullet \in \B'_{e'}$. The standard tableau $\sigma^{(1)}{}',...,\sigma^{(j)}{}'$ corresponding to $F^{(1)}{}'_\bullet,...,F^{(j)}{}'_\bullet$ are determined from $\sigma^{(1)},...,\sigma^{(j)}$ in the obvious manner. The tableau $\sigma^{\Sigma}{}' := \stk(\sigma^{(1)}{}',...,\sigma^{(j)}{}')$ is the numbered diagram obtained from $\sigma^\Sigma$ by considering only the boxes labelled $1,...,i$. Since we have assumed $\sigma_i^\Sigma \ne \tau_i$, it follows that $\sigma_i^{\Sigma}{}' \ne \tau_i'$. To prove the Proposition we must show that $\sigma_i^\Sigma{}' > \tau'_i$.\\

Since we are in precisely the same situation as the start of the proof, we might as well simplify our notation by assuming that $i = N$, and working with $\g, \p, e, \g_0, e_0$ rather than $\g', \p', e', \g_0', e_0'$. We shall assume $\sigma_N^\Sigma \ne \tau_N$ to prove $\sigma_N^\Sigma > \tau_N$. \\

Suppose that $\lambda_n$ is in the $(i,j)$-th position of $\sigma^{(n)}$ (so we have now fixed $i$ and $j$). Then by the formula in \eqref{eq:stack}, $N$ occupies the $(i,j')$-th position in $\sigma^{\Sigma}$, where $ j' = \sum_{k=1}^{n-1} \mu_{k,i} + j$; that is, $\sigma_{N}^{\Sigma}=j'$. Hence we need to show that $\bar{F}_{N-1} \supseteq \ker(e^{j'-1})$ to deduce that $\tau_{N} \geq j' =\sigma_{N}^{\Sigma}$. From the outset we assumed that $\tau_N \ne \sigma^\Sigma_N$ and so this will imply that $\tau > \sigma^\Sigma$ and complete the proof.\\

By Lemma~\ref{L:LSonflags}, we have $\oF_{N-1} = W_{n-1} \oplus F^{(n)}_{\lambda_{n}-1}$ and we have that $j$ is the maximal index such that $F^{(n)}_{\lambda_n-1} \supseteq \ker(e_{0|V_{n}}^{j-1})$. At this point, we use Proposition \ref{L:reductionlemma} to make an explicit choice for $e_0$ and $e_1$ that comes with a Jordan basis as in Section~\ref{ss:representative}. Thus we picture the $N$-dimensional vector space as the span of all the boxes in the tuple of Young diagrams corresponding to the partitions $(\mu_1, \ldots, \mu_{n})$, and hence we may identify $W_{n-1}$ as the span of the boxes in the first $n-1$ Young diagrams and $F_{\lambda_{n}-1}^{(N)}$ as the span of all the boxes in the last Young diagram except for the box in the $(i,j)$-th position (or the box at the bottom of column $j$ in last Young diagram). This pictorial interpretation is explained clearly in Example~\ref{ex:representativediagram}. \\

Since may we diagrammatically represent $\ker(e^{j'-1})$ as the span of the boxes in the first $j'-1$ columns of the stacked Young diagram (see Diagram 2 in Section~\ref{ss:representative}), it follows that $\ker(e^{j'-1}) \subseteq W_{n-1} \oplus \ker(e_{0|V_{n}}^{j-1})$. Since $\ker(e_{0|V_{n}}^{j-1}) \subseteq F^{(n)}_{\lambda_n-1}$ by construction, it follows that 
$$
\ker(e^{j'-1})  \subseteq W_{n-1} \oplus  F^{(n)}_{\lambda_{n}-1}= \oF_{N-1}.
$$
Therefore $\tau_{N} \geq j' =\sigma_{N}^{\Sigma}$ and so $\tau > \sigma^{\Sigma}$, and the proof is complete. 
\end{proof}

The following example illustrates that $\LS(X_{\sigma_1,...,\sigma_n})$ is not contained in $X_{\stk(\sigma_1,....,\sigma_n)}$ in general, and so the previous proposition is best possible.
\begin{example}
Let $\k^5$ have basis $\{v_1, v_2, v_3, v_4, v_5\}$ and let $V_1 = \Span\{v_1, v_2, v_3\}$ and $V_2 = \Span\{ v_4, v_5\}$. Let $e_{i,j}$ be the standard basis for $\g = \gl_5$ with respect to this basis of $\k^5$ and let $\g_0$ be the Levi factor of $\g$ preserving the decomposition $\k^5 = V_1\oplus V_2$. Let $\p$ be the parabolic preserving the flag $0 \subseteq V_1 \subseteq \k^5$. We define $e_0 := e_{1,2}$ and $e_1 = e_{3,4} + e_{2,5}$.  One can check that these data satisfy the set up of Proposition~\ref{prop:LSmap}.

Then if $F^{(1)}$ is any full flag of $V_1$ preserved by $e_0$, satisfying $e_0|_{F^{(1)}_{2}} = 0$, and $$F^{(2)} = (0 \subseteq \k v_5 \subseteq V_2)$$ then $(F^{(1)}, F^{(2)}) \in \B^0_{e_0}$. On the other hand, $\LS(F^{(1)}, F^{(2)})$ lies in $X_\tau$ where
$$\tau = \young(135,24)$$
However the stacked tableau $\sigma$ corresponding to $F^{(1)}, F^{(2)}$ is the one appearing in \eqref{e:atab}. In the notation of Section~\ref{subsection:ordering} we have $\tau_5 = 3 > 2 = \sigma_5$ and so $\tau \gneq \sigma$.
\end{example}

\subsection{The stacking theorem}

Retain the notation and choices made at the start of Section~\ref{ss:LSonfibres}. Let $\sigma^{(1)}, \ldots, \sigma^{(n)}$ be a choice of standard tableaux of shape $\mu_1,...,\mu_n$ respectively. 
By Spaltenstein's theorem, the closure $\overline{X}_{\sigma^{(1)}, \ldots, \sigma^{(n)}}$ is an irreducible component of $\B^0_{e_0}$ which we denote $C_{\sigma^{(1)}, \ldots, \sigma^{(n)}}$. Similarly for $\sigma$ a standard tableau of shape $\mu^\Sigma$ we write $C_\sigma := \overline{X}_\sigma \subseteq \B_e$, which is an irreducible component of $\B_e$.

The following is our main result.
\begin{thm}
\label{T:stackwithproof}
$\LS(C_{\sigma^{(1)}, \ldots, \sigma^{(n)}}) = C_{\stk(\sigma^{(1)}, \ldots, \sigma^{(n)})}$.
\end{thm}
\begin{proof}
The first step is to show that there is an element $(F^{(1)}_\bullet,...,F^{(n)}_\bullet) \in \B^0_{e_0}$ such that $\LS(F^{(1)}_\bullet,...,F^{(n)}_\bullet) \in X_{\stk(\sigma^{(1)},...,\sigma^{(n)})}$. Notice that using the argument in the first paragraph of the proof of Proposition~\ref{prop:LSmap} it suffices to find such a flag for a particular choice of $e\in \O_\p$, and it follow for all such elements. Pick a basis $\{v_{j,l,m} \mid j, l, m\}$ for $\k^N$ such that $\{v_{j,l,m} \mid l, m\}$ spans the spaces $V_j$ described in Section~\ref{ss:LSonfibres} and then choose an element $e = e_0 + e_1 \in \O_\p$ using formulas \eqref{e:e0defn} and \eqref{e:e1defn}. For each $k=1,...,n$ we define a full flag $F^{(k)}$ of $V_j$ by letting $F^{(k)}_j$ be the span of those vectors $v_{k,l,m}$ such that the $(l,m)$-entry of $\sigma^{(k)}$ is less than or equal to $j$.\vspace{6pt}

\noindent {\bf Claim.} \ $(F^{(1)}_\bullet, \ldots, F^{(n)}_\bullet) \in \B^0_{e_0}$ and $\LS(F^{(1)}_\bullet, \ldots, F^{(n)}_\bullet) \in X_{\stk(\sigma^{(1)},...,\sigma^{(n)})}$.\vspace{6pt}

The first part of the claim is obvious from the construction. To prove the second part we describe the Jordan blocks of $e|_{\oF_j}$ where $\oF_{\bullet} := \LS(F^{(1)}_\bullet, ..., F^{(n)}_\bullet)$. For $j = 1,...,N$ we let $k \ge 0$ be such that $\sum_{i=1}^{k-1} \lambda_i < j \le \sum_{i=1}^{k} \lambda_i$. Write
$$\cX_j := \{(i,l,m)\mid i < k \text{ or } i = k  \text{ and the } (l,m)\text{-entry of } \sigma^{(k)} \text{ is less than or equal to }\tilde j := j - \sum_{r=1}^{k-1} \lambda_r\}.$$
Then $\oF_j$ is spanned by $\{v_{i,l,m} \mid (i,l,m) \in \cX_j\}$. Now if we fix $1\le i \le \max \ell(\mu_s)$ then the span of $\{v_{i,l,m} \mid (i,l,m) \in \cX_j\}$ is a single Jordan block for $e$. We invite the reader to check this claim using Diagram 2 in Example~\ref{ex:representativediagram}. Combining this description of the Jordan blocks of $e|_{\oF_k}$ with the relationship between standard tableaux and sequences of nested partitions given in Section~\ref{ss:LSonfibres} it follows that $\Phi(\oF_\bullet) = \stk(\sigma^{(1)},...,\sigma^{(n)})$. This proves the claim.

By Corollary~\ref{cor:inducedmaponcomponents} and Proposition~\ref{prop:LSmap} we know that $\LS(\overline{X}_\sigma^{(1)},...,\sigma^{(n)})$ is equal to $\overline{X}_{\tau}$ for some $\tau \ge \stk(\sigma^{(1)},...,\sigma^{(n)})$. Combining the above Claim with Lemma~\ref{lem:Xsigmaintersectsclosure} we deduce that $\tau = \stk(\sigma^{(1)},...,\sigma^{(n)})$, which completes the proof.
\end{proof}

\begin{rem}
A direct consequence of the stacking theorem is that $\stk$ is an associative operation on tableaux, which can also be checked directly by a combinatorial argument.
\end{rem}

\bibliographystyle{alpha}
\bibliography{LS-bibliography}

\end{document}